\documentclass[11pt]{article}
\usepackage[margin=1.2in]{geometry}
\usepackage{amsmath,amssymb,amsthm}
\usepackage[colorlinks=true,linkcolor=blue,citecolor=blue,urlcolor=blue]{hyperref}

\newtheorem{theorem}{Theorem}
\newtheorem{lemma}[theorem]{Lemma}
\newtheorem{corollary}[theorem]{Corollary}
\newtheorem{proposition}[theorem]{Proposition}
\theoremstyle{remark}
\newtheorem{remark}[theorem]{Remark}

\newcommand{\A}{\mathcal{A}}

\DeclareMathOperator{\lcm}{lcm}
\DeclareMathOperator{\dens}{dens}

\title{The density of sums of distinct divisors}
\author{Scott D. Hughes}
\date{\today}

\begin{document}
\maketitle

\begin{abstract}
For a positive integer $t$, let $d_t$ denote the natural density of the set of $n$ for which
$t$ is a sum of distinct divisors of $n$. Erd\H{o}s proved that $d_t$ exists, gave an
unspecified polylogarithmic upper bound, asserted without proof a matching lower bound, and
asked whether $d_t \sim c_3/(\log t)^{c_4}$. We record the explicit bounds
\[
  \nu_t \;\le\; d_t \;\ll\; \frac{(\log\log t)^{\delta-3/2}}{(\log t)^{\delta}},
  \qquad
  \nu_t = \frac{K}{\log t}\Bigl(1 + O\Bigl(\frac{\log\log t}{\log t}\Bigr)\Bigr),
\]
where $\delta = 0.086071\ldots$ is the Erd\H{o}s--Ford--Tenenbaum constant,
$K=c\,e^{-\gamma}$, and $c=1.33607\ldots$ is the practical-number constant. Consequently, if
Erd\H{o}s's asymptotic holds, then $\delta<c_4\le 1$. A two-prime construction, using a
half-scale sumset to obtain full residue coverage, then yields the pointwise excess
\[
  \liminf_{t\to\infty}(\log t)\,(d_t-\nu_t) \;\ge\; KI,
\]
where $I=\int_0^2 G(w)\,\mathrm{d}w=0.05887\ldots$ is an explicit elementary integral. In
particular $d_t\ge 0.79/\log t$ for every sufficiently large $t$, and $d_t$ is not
asymptotic to $K/\log t$.
\end{abstract}

\section{Introduction}\label{sec:intro}

\paragraph{Conventions.} Throughout, $t$, $n$, $m$ denote positive integers, $p$, $q$ primes,
and $\log$ the natural logarithm. We write $P^-(n)$ and $P^+(n)$ for the smallest and largest
prime factors of $n$, with $P^-(1)=\infty$ and $P^+(1)=1$, and $\sigma$, $\tau$ for the sum
and number of divisors; $\gamma$ is Euler's constant.

Say $t$ is \emph{representable from} $n$ if $t=\sum_{d\in S}d$ for a set $S$ of distinct
divisors of $n$, and let
\[
  \A_t := \{n : t \text{ is representable from } n\}, \qquad
  d_t := \text{the natural density of }\A_t,
\]
which exists and is rational for every $t$ (Lemma~\ref{lem:periodic};
cf.\ \cite[p.~130]{Er70}). Erd\H{o}s \cite[p.~130]{Er70} proved $d_t<(\log t)^{-c_1}$ for an
unspecified $c_1>0$, stated without proof a lower bound of the same shape, and asked (his
eq.~(33)) whether
\begin{equation}\label{eq:erdos-conj}
  d_t = \frac{(1+o(1))\,c_3}{(\log t)^{c_4}}
\end{equation}
holds, adding that this ``if true may not be quite easy to prove.'' The problem is recorded
as Erd\H{o}s Problem \#859 \cite{Bl859}.

\begin{remark}[a notation clash with the problem list]\label{rem:notation}
We follow Erd\H{o}s's own numbering from \cite[p.~130]{Er70}: $c_1,c_2$ are the exponents in
his two-sided bounds, and $c_3,c_4$ are the constant and the exponent in the conjectured
asymptotic \eqref{eq:erdos-conj}. The problem list \cite{Bl859} uses the two pairs in the
opposite roles, writing $d_t\sim c_1/(\log t)^{c_2}$ for the asymptotic and
$(\log t)^{-c_3}<d_t<(\log t)^{-c_4}$ for the bounds. In particular our $c_4$ is the $c_2$ of
\cite{Bl859}.
\end{remark}

The order-of-magnitude lower bound $d_t\gg 1/\log t$ is an immediate consequence of Pollack
and Thompson's theorem on practical pretenders \cite{PoTh13}. The first two results make both
sides explicit. Write $K=c\,e^{-\gamma}$, where $c=1.33607\ldots$ is the practical-number
constant \cite{We15,We20}.

\begin{theorem}\label{thm:lower}
Let $f(n)$ be the largest nonnegative integer $Y$ such that every integer in $[1,Y]$ is a
sum of distinct divisors of $n$, and let $\nu_t$ denote the density of
$\{n:f(n)\ge t\}$. Then $d_t\ge\nu_t$ for every $t$, and, as $t\to\infty$,
\begin{equation}\label{eq:lower}
  \nu_t \;=\; \frac{K}{\log t}
  \left(1 + O\!\left(\frac{\log\log t}{\log t}\right)\right).
\end{equation}
In particular $\liminf_{t\to\infty}\, d_t\log t \;\ge\; K \;=\; 0.75015\ldots$.
\end{theorem}

\begin{theorem}\label{thm:upper}
As $t\to\infty$,
\[
  d_t \;\ll\; \frac{(\log\log t)^{\delta-3/2}}{(\log t)^{\delta}},
  \qquad
  \delta = 1 - \frac{1+\log\log 2}{\log 2} = 0.086071\ldots.
\]
In particular $d_t=o\bigl((\log t)^{-\delta}\bigr)$.
\end{theorem}

\begin{corollary}\label{cor:exponent}
If \eqref{eq:erdos-conj} holds with constants $c_3>0$ and $c_4$, then $\delta<c_4\le 1$;
and if $c_4=1$, then $c_3\ge K$.
\end{corollary}

Theorem~\ref{thm:lower} follows from Weingartner's theorem on integers with large practical
component \cite{We15pc}; Theorem~\ref{thm:upper} combines Erd\H{o}s's window decomposition
with Ford's theorem on divisors in intervals \cite{Fo08}. We claim no novelty of method in
either.

A natural guess, refining \eqref{eq:erdos-conj}, is that the lower bound is asymptotically
sharp: $d_t\sim K/\log t$. Equivalently, the density of integers from which one target $t$
is representable would match the density of integers from which every target up to $t$ is
representable. This is false, and the failure is pointwise.

Recall that $m$ is \emph{practical} if every integer in $[1,m]$ is a sum of distinct
divisors of $m$. Every $n$ factors uniquely as $n=mr$ with $m$ practical and
$P^-(r)>\sigma(m)+1$ (Lemma~\ref{lem:fact}). Grouping the divisors in a representation of
$t$ by their rough cofactor, the event that two distinct primes $p,q$ both divide the
cofactor and $t=b+ap+cq$ with bounded coefficients already produces a definite excess over
$\nu_t$. The residue estimate for this event upgrades, by a half-scale sumset, from covering
half of $\mathbb F_p$ to covering all of $\mathbb F_p$. Combined with an exact
finite-divisibility probability over a range $\sigma(m)\in(t^{1/3},t)$, this yields a
limiting profile $G$ and an explicit integral $I=\int_0^2 G$.

\begin{theorem}\label{thm:pointwise}
Unconditionally,
\[
  \liminf_{t\to\infty}(\log t)\,(d_t-\nu_t) \;\ge\; KI,
  \qquad
  \liminf_{t\to\infty}(\log t)\,d_t \;\ge\; K(1+I),
\]
where $I$ is the elementary integral \eqref{eq:I} below, with
$I=0.0588797799\ldots$. In particular $d_t\ge 0.79/\log t$ for every sufficiently large
integer $t$.
\end{theorem}

The second lower-bound constant is approximately $0.79432$. It is not asserted to be an
asymptotic constant for $d_t$, nor is $c_4=1$ established. The averaged and
positive-proportion excesses of an earlier draft are immediate, weaker consequences.

\begin{corollary}\label{cor:average}
There is an absolute constant $\beta_0>0$ such that, for all sufficiently large $X$,
\[
  \frac1X\sum_{X\le t<2X}\bigl(d_t-\nu_t\bigr) \;\ge\; \frac{\beta_0}{\log X}.
\]
The set of $t$ with $(\log t)\,(d_t-\nu_t)\ge\beta_0/2$ has lower density $1$.
\end{corollary}

\begin{corollary}\label{cor:refuted-constant}
The constant $K$ is not the asymptotic constant for $d_t\log t$. If \eqref{eq:erdos-conj}
holds, then either $c_4<1$, or $c_4=1$ with $c_3\ge K(1+I)>0.79$.
\end{corollary}

The rest of the note is organised as follows. Section~\ref{sec:framework} records the
periodicity of $\A_t$ and the practical-component factorisation.
Section~\ref{sec:bounds} proves Theorems~\ref{thm:lower} and~\ref{thm:upper}.
Section~\ref{sec:excess} develops the two-prime construction and proves
Theorem~\ref{thm:pointwise}. Section~\ref{sec:constant} evaluates $I$.

\section{Density and the practical component}\label{sec:framework}

That $d_t$ exists for every $t$ is due to Erd\H{o}s \cite{Er70}. The following lemma records
the elementary periodic structure, which additionally gives rationality with an explicit
denominator.

\begin{lemma}\label{lem:periodic}
For every positive integer $t$,
\[
  \A_t \;=\; \bigcup_{\substack{S\subseteq\{1,\ldots,t\}\\ \sum_{a\in S}a=t}}
  \{\,n:\lcm(S)\mid n\,\}.
\]
Consequently $\A_t$ is a finite union of residue classes modulo $L_t:=\lcm(1,\ldots,t)$,
and $d_t$ exists and is rational with denominator dividing $L_t$.
\end{lemma}

\begin{proof}
If $n\in\A_t$, the representing set $S$ consists of distinct divisors of $n$, each at most
$t$, and $\lcm(S)\mid n$. Conversely, if $S\subseteq\{1,\ldots,t\}$ has sum $t$ and
$\lcm(S)\mid n$, every element of $S$ divides $n$. Each $\lcm(S)$ divides $L_t$.
\end{proof}

By a theorem of Stewart \cite{St54} and Sierpi\'nski \cite{Si55},
$m=p_1^{a_1}\cdots p_k^{a_k}$ with $p_1<\cdots<p_k$ is practical if and only if $p_1=2$ and
$p_{j+1}\le\sigma(p_1^{a_1}\cdots p_j^{a_j})+1$ for $1\le j<k$. Moreover $f(n)=\sigma(g(n))$,
where $g(n)$ is the largest practical divisor of $n$ \cite[Lemma~2.1]{PoTh13}. Every
practical $m\ge 2$ is even, and $\sigma(m)/m\ge 3/2$ for every practical $m\ge 2$. Also
$\sigma(m)/m\ll\log\log m$ for all $m\ge 3$ (Gronwall; see \cite[Thm.~323]{HaWr08}).

Weingartner \cite[Cor.~1]{We15pc} proved that $\{n:g(n)=m\}$ has density
\begin{equation}\label{eq:alpha}
  \alpha_m \;=\; \frac{\chi(m)}{m}\prod_{p\le\sigma(m)+1}\Bigl(1-\frac1p\Bigr),
\end{equation}
$\chi$ the indicator of practicality. The unique decomposition of every integer into a
practical component and a rough cofactor is standard --- it is stated explicitly in the
proof of Lemma~1 of \cite{We15pc}, with the practical component as in \cite{PoTh13}.
Henceforth sums over $m$ are restricted to practical integers.

\begin{lemma}\label{lem:fact}
$g(n)=m$ if and only if $m\mid n$, $m$ is practical, and $P^-(n/m)>\sigma(m)+1$.
Consequently every $n$ admits a unique factorisation $n=mr$ with $m$ practical and
$P^-(r)>\sigma(m)+1$; the divisors of $n$ are exactly the products $ef$ with $e\mid m$,
$f\mid r$, and this representation of each divisor is unique. Moreover, for each fixed
practical $m$ and each $f$ with $P^-(f)>\sigma(m)+1$,
\begin{equation}\label{eq:cond}
  \#\{n\le x:g(n)=m,\ f\mid(n/m)\}
  \;=\; \#\{r\le x/(mf):P^-(r)>\sigma(m)+1\}.
\end{equation}
\end{lemma}

\begin{proof}
First, if $m$ is practical and $q\le\sigma(m)+1$ is prime, then $mq$ is practical. For
$m=1$ the only such prime is $q=2$ and $mq=2$ is practical; so assume $m\ge 2$, hence
$2\mid m$. If $q\mid m$ this raises one exponent, which weakens no chain condition; if
$q>P^+(m)$ the chain is extended by one step, legal since $q\le\sigma(m)+1$; and if
$q<P^+(m)$ with $q\nmid m$, let $p^*$ be the least prime factor of $m$ exceeding $q$ and
$P$ the product of the prime powers in $m$ below $q$ --- then $q<p^*\le\sigma(P)+1$ by the
criterion for $m$, so $q$ may be inserted, and all later chain conditions only improve as
$\sigma$ grows. Hence if $g(n)=m$ and some prime $q\le\sigma(m)+1$ divided $n/m$, then
$mq\mid n$ would be a larger practical divisor of $n$, contradicting maximality.

Conversely, suppose $P^-(n/m)>\sigma(m)+1$ with $m$ practical, and let $m'\mid n$ be
practical. Every prime factor of $m'$ not exceeding $\sigma(m)+1$ divides $m$, with
$v_p(m')\le v_p(n)=v_p(m)$. If $m'$ had a prime factor exceeding $\sigma(m)+1$, let $Q$ be
the least such: then every prime factor of $m'$ below $Q$ is at most $\sigma(m)+1$, so the
chain prefix $P'$ of $m'$ below $Q$ divides $m$, whence $Q\le\sigma(P')+1\le\sigma(m)+1$, a
contradiction. So $m'\mid m$, and $m$ is the largest practical divisor.

Uniqueness and the divisor decomposition follow since $\gcd(m,r)=1$ (all prime factors of
$r$ exceed $\sigma(m)+1>P^+(m)$). Identity \eqref{eq:cond} is immediate from the first
claim applied to $n=mfr'$: for $n\le x$ with $mf\mid n$ and $n/(mf)=r'$, the condition
$g(n)=m$ holds if and only if $P^-(fr')>\sigma(m)+1$, and since $P^-(f)>\sigma(m)+1$ is
given, this is the condition $P^-(r')>\sigma(m)+1$ on $r'\le x/(mf)$.
\end{proof}

Taking $x\to\infty$ in \eqref{eq:cond} yields, for $P^-(f)>\sigma(m)+1$,
\begin{equation}\label{eq:dens-mf}
  \dens\{n:g(n)=m,\ f\mid n/m\} \;=\; \frac{\alpha_m}{f}.
\end{equation}
The event $\{f(n)<y\}$ is the disjoint union of $\{g(n)=m\}$ over practical $m$ with
$\sigma(m)<y$. This union is finite, because $m\le\sigma(m)<y$. Natural density is finitely
additive, so
\begin{equation}\label{eq:tail}
  \nu_y \;=\; 1-\sum_{\sigma(m)<y}\alpha_m.
\end{equation}
(Weingartner \cite[p.~441]{We15pc} records the finite-complement formula explicitly.) No
countable additivity of density is used.

The subset sums of a practical number fill an interval exactly.

\begin{lemma}\label{lem:subsetsums}
If $m$ is practical, the subset sums of the divisors of $m$ are exactly the integers in
$[0,\sigma(m)]$.
\end{lemma}

\begin{proof}
Suppose the divisors of $M$ have subset sums filling $[0,S]$, $S=\sigma(M)$, and let
$p\le S+1$ be a prime not dividing $M$, $a\ge 1$. The divisors of $Mp^a$ are the disjoint
blocks $p^j\{d:d\mid M\}$, $0\le j\le a$; choosing a subset independently in each block
realizes every sum $\sum_{j=0}^{a}p^j u_j$ with $0\le u_j\le S$. Increasing the digit
$u_j$ by one and letting the lower digits range fills consecutive intervals which overlap
because $p^j\le 1+S(1+p+\cdots+p^{j-1})$, an inequality that holds since $S\ge p-1$. Hence
the sums fill $[0,S\sigma(p^a)]=[0,\sigma(Mp^a)]$. By the Stewart--Sierpi\'nski criterion,
every practical $m$ is built from $M=1$ by such steps taken one prime power at a time (the
first step adjoins $2^{a_1}$, legal as $2\le\sigma(1)+1$).
\end{proof}

\section{Proofs of the two-sided bounds}\label{sec:bounds}

\begin{proof}[Proof of Theorem~\ref{thm:lower}]
If $f(n)\ge t$ then in particular the integer $t$ is a sum of distinct divisors of $n$, so
$\{n:f(n)\ge t\}\subseteq\A_t$. For $t=1$, $f(n)\ge 1$ for all $n$ and $d_1=\nu_1=1$. For
fixed $t\ge 2$, Weingartner \cite[Thm.~1(iv)]{We15pc} proved
$\#\{n\le x:f(n)\ge t\}=x\nu_t+O(2^t)$; dividing by $x$ and letting $x\to\infty$ (both
limits exist, by Lemma~\ref{lem:periodic} for the right side) gives $d_t\ge\nu_t$, and the
display following his eq.~(3) gives the asymptotic \eqref{eq:lower}.
\end{proof}

\begin{proof}[Proof of Theorem~\ref{thm:upper}]
We follow Erd\H{o}s's splitting \cite[p.~130]{Er70}. (Erd\H{o}s wrote the window with an
open upper endpoint, $(t/(\log t)^2,\,t)$; we use $(t/(\log t)^2,\,t]$.) Put $L:=\log t$,
$y:=t/L^2$, $z:=t$, and let $t$ be large. Suppose $n\in\A_t$, say $t=d_1+\cdots+d_k$ with
$d_1>\cdots>d_k$ distinct divisors of $n$; every $d_i\le t$.

\emph{Case 1: some $d_i>y$.} Then $n$ has a divisor in $(y,z]$. For large $t$ we have
$y\ge 100$, $2y\le z$, and $z\le y^2$. For $1\le y\le z\le x$ let $H(x,y,z)$ denote the
number of $n\le x$ with a divisor in $(y,z]$, and
$\varepsilon(y,z):=\lim_{x\to\infty}H(x,y,z)/x$, which exists by periodicity of the
underlying set. Ford \cite[Thm.~1(v)]{Fo08} gives, in the range $100\le y$,
$2y\le z\le y^2$, writing $z=y^{1+u}$,
\[
  H(x,y,z) \;\asymp\; x\, u^{\delta}\Bigl(\log\frac2u\Bigr)^{-3/2}
  \qquad (x>x_0,\ y\le\sqrt{x}),
\]
uniformly. Here
\[
  u \;=\; \frac{\log(z/y)}{\log y} \;=\; \frac{2\log\log t}{\log t-2\log\log t}
  \;\asymp\; \frac{\log\log t}{\log t},
  \qquad
  \log\frac2u \;\asymp\; \log\log t,
\]
so, dividing by $x$ and letting $x\to\infty$ for each fixed $t$,
\[
  \varepsilon(y,z)
  \;\ll\; \Bigl(\frac{\log\log t}{\log t}\Bigr)^{\!\delta}(\log\log t)^{-3/2}
  \;=\; \frac{(\log\log t)^{\delta-3/2}}{(\log t)^{\delta}}.
\]

\emph{Case 2: every $d_i\le y$.} Write $S_y(n)=\sum_{d\mid n,\,d\le y}d$. A representation
using only divisors at most $y$ forces $S_y(n)\ge t$, while
$\sum_{n\le x}S_y(n)=\sum_{d\le y}d\lfloor x/d\rfloor\le x\lfloor y\rfloor$. The density of
this event is therefore at most $y/t=1/(\log t)^2$.

Combining, $d_t\le\varepsilon(y,z)+O\bigl((\log t)^{-2}\bigr)
\ll(\log\log t)^{\delta-3/2}(\log t)^{-\delta}$, since $\delta<1$. As $\delta-3/2<0$, the
bound is $o((\log t)^{-\delta})$.
\end{proof}

\begin{proof}[Proof of Corollary~\ref{cor:exponent}]
By Theorem~\ref{thm:lower}, $\liminf_{t\to\infty}d_t\log t\ge K>0$; if $c_4>1$ the
asymptotic \eqref{eq:erdos-conj} would force $d_t\log t\to 0$, so $c_4\le 1$. By
Theorem~\ref{thm:upper}, $d_t(\log t)^{\delta}\to 0$; if $c_4<\delta$ the asymptotic would
force $d_t(\log t)^{\delta}\to\infty$, and if $c_4=\delta$ it would force
$d_t(\log t)^{\delta}\to c_3>0$, so $c_4>\delta$. If $c_4=1$ then $d_t\log t\to c_3$, and
the lower limit gives $c_3\ge K$.
\end{proof}

\begin{remark}
Ford's estimate is sharp ($\asymp$) for the unconditioned set of integers having some
divisor in the window, so no improvement in Theorem~\ref{thm:upper} can come from that
estimate alone: one must exploit the requirement that the window divisor participates in an
actual representation of $t$.
\end{remark}

\section{The two-prime excess}\label{sec:excess}

Write $n=mr$ as in Lemma~\ref{lem:fact}, and $s:=\sigma(m)$. Each divisor of $n$ is uniquely
$ef$ with $e\mid m$ and $f\mid r$. Given a representation $t=\sum_{d\in S}d$ from $n$, group
the elements of $S$ by their rough cofactor: for each $f\mid r$ let $a_f$ be the sum of the
$m$-parts $e$ of those elements of $S$ with cofactor $f$ (with $a_f=0$ if none), so that
$t=a_1+\sum_{f>1}a_f f$.

For $s\ge 1$ let $U_2^{\mathrm{pr}}(t,s)$ denote the density of integers $r$ for which there
are distinct primes $p,q>s+1$, both dividing $r$, and integers
\begin{equation}\label{eq:twoprime-event}
  t = b + ap + cq, \qquad 1\le a,c\le s, \quad 0\le b\le s.
\end{equation}
Only primes below $t$ can occur, so this event is periodic of period dividing
$\lcm(1,\ldots,t)$. The phrase ``two active prime cofactors'' refers to this
representation, not to $\omega(r)$.

\begin{lemma}\label{lem:reduction}
For every $t\ge 1$,
\begin{equation}\label{eq:reduction}
  d_t-\nu_t \;\ge\; \sum_{\sigma(m)<t}\alpha_m\, U_2^{\mathrm{pr}}\bigl(t,\sigma(m)\bigr).
\end{equation}
The sum is finite.
\end{lemma}

\begin{proof}
Fix practical $m$ with $s=\sigma(m)<t$. Conditional on $g(n)=m$, the cofactor $r=n/m$ is
constrained only in that $P^-(r)>s+1$. For primes $p,q>s+1$, divisibility by $p$ and $q$
retains its ordinary Chinese-remainder probabilities, independently of the small-prime
constraint, so the density of $n$ with $g(n)=m$ for which \eqref{eq:twoprime-event} holds
is $\alpha_m\,U_2^{\mathrm{pr}}(t,s)$. These events for distinct $m$ are disjoint.

If $g(n)=m$ and \eqref{eq:twoprime-event} holds, then $f(n)=s<t$. Lemma~\ref{lem:subsetsums}
realises each of the three numerical coefficients $a,b,c$ separately as a subset sum of
divisors of $m$ (combine the coefficients first; do not union previously chosen subsets,
which might overlap). The resulting products with cofactors $1$, $p$, and $q$ are pairwise
distinct, so $n\in\A_t$. Thus each such $n$ lies in $\A_t\setminus\{f\ge t\}$, a set of
density $d_t-\nu_t$.
\end{proof}

\subsection{From energy to full residue coverage}

\begin{lemma}[Rectangular energy]\label{lem:energy}
Let $p$ be prime and $1\le M,N<p$. For $u\in\mathbb F_p^\times$, let
\[
  R_u(x)=\#\{(c,b):1\le c\le M,\ 0\le b<N,\ cu+b\equiv x\pmod p\},
  \qquad E_u=\sum_{x\in\mathbb F_p}R_u(x)^2.
\]
Then
\begin{align}
\sum_{u\ne 0}E_u&=(p-1)MN+M(M-1)N(N-1),\label{eq:energy1}\\
\sum_{u\ne 0}\Bigl(E_u-\frac{(MN)^2}{p}\Bigr)
&=\frac{MN(p-M)(p-N)}{p}.\label{eq:energy2}
\end{align}
Fewer than $(p-M)(p-N)/(MN)$ slopes have $|\operatorname{supp} R_u|<p/2$.
\end{lemma}

\begin{proof}
An identical ordered pair of grid points contributes $p-1$ slopes. If just one coordinate
agrees, it contributes no nonzero slope. If both coordinates differ, it contributes exactly
one nonzero slope. This proves \eqref{eq:energy1}; algebra gives \eqref{eq:energy2}. Every
summand in \eqref{eq:energy2} is nonnegative by Cauchy--Schwarz. If the support has size
below $p/2$, then $E_u>2(MN)^2/p$, so the corresponding summand exceeds $(MN)^2/p$. The
claimed exceptional count follows.
\end{proof}

\begin{lemma}[Full coverage]\label{lem:full}
Suppose $S,C\ge 2$ are integers, $p$ is an odd prime, and $p>\max(C,S+1)$. Except for
$O(p^2/(CS))$ slopes $u\in\mathbb F_p^\times$,
\[
  \{cu+b\pmod p:1\le c\le C,\ 0\le b\le S\}=\mathbb F_p.
\]
The implied constant is absolute.
\end{lemma}

\begin{proof}
Apply Lemma~\ref{lem:energy} with $M=\lfloor C/2\rfloor$ and $N=\lfloor S/2\rfloor+1$.
Outside at most $p^2/(MN)\ll p^2/(CS)$ exceptional slopes, the associated support $A$ has
size at least $p/2$, hence strictly greater than $p/2$ because $p$ is odd. For every
$z\in\mathbb F_p$, the sets $A$ and $z-A$ intersect, so $A+A=\mathbb F_p$. Each sum has
coefficient $2\le c\le 2M\le C$ and $0\le b\le 2N-2\le S$.
\end{proof}

In particular, taking $C=S=s$ already upgrades the half-residue estimate of a square window
to full coverage of $\mathbb F_p$. The argument below uses a rectangular window to
accommodate larger $q$.

\subsection{Admissible pairs}

Fix a compact interval $J\subset(1/3,1)$. Throughout this subsection
\[
  \ell=\log t, \qquad \kappa=\frac{\log s}{\log t}\in J,
\]
where $s$ is an integer, and all error bounds are uniform for $\kappa\in J$. Set
\begin{equation}\label{eq:windows}
  L_0=\max(2s,4t/s), \qquad
  P_0=\frac{\sqrt{ts}}{\ell^4}, \qquad
  \mathcal H=\{p\text{ prime}:L_0<p\le P_0\}.
\end{equation}
For each $p\in\mathcal H$, put
\begin{equation}\label{eq:B}
  B(p)=\max\Bigl(2s,\frac{p^2\ell^4}{s^2}\Bigr).
\end{equation}
An admissible pair is $p\in\mathcal H$ and a prime $B(p)<q<p$. The interval $\mathcal H$
has positive logarithmic length for large $t$, since
\[
  \max(\kappa,1-\kappa)<\frac{1+\kappa}{2}\qquad (1/3<\kappa<1).
\]
Moreover $B(p)\le L_0$, because $p^2\ell^4/s^2\le t/(s\ell^4)<4t/s$. Both primes of an
admissible pair exceed $s+1$, and every such $p$ is odd.

\begin{proposition}\label{prop:bad}
Apart from a set of integers $r$ of density $O(\ell^{-3})$, every $r$ divisible by an
admissible pair belongs to the event defining $U_2^{\mathrm{pr}}(t,s)$.
\end{proposition}

\begin{proof}
For fixed $p$, split $B(p)<q<p$ into intersections with dyadic intervals $Q<q\le 2Q$. There
are $O(\ell)$ intervals. Let $C_Q=\min(s,t/(8Q))$. This tends to infinity uniformly on $J$,
since $Q<p\le P_0$ and $\kappa$ is bounded away from $1$. Apply Lemma~\ref{lem:full} with
coefficient limit $\lfloor C_Q\rfloor$ and $S=s$. The exceptional count is
$O(p^2/(s C_Q))$. Distinct primes $q<p$ occupy distinct nonzero residues modulo $p$, so
their exceptional reciprocal mass in this interval is
\[
  \sum_{\substack{Q<q\le 2Q\\ q\text{ exceptional}}}\frac1q
  \;\ll\; \frac{p^2}{s C_Q Q}
  \;\ll\; \frac{p^2}{s^2 Q}+\frac{p^2}{st}
  \;\ll\; \ell^{-4}+\ell^{-8},
\]
using $Q>B(p)/2$ and $p\le P_0$. Thus the bad reciprocal mass for a fixed $p$ is
$O(\ell^{-3})$.

For every nonexceptional $q$, full coverage gives integers $1\le c\le C_Q$ and $0\le b\le s$
with $t\equiv cq+b\pmod p$. They satisfy $cq\le t/4$ and $cq+b\le t/4+s<t$ for large $t$.
Hence
\[
  a=\frac{t-cq-b}{p}\in\mathbb Z, \qquad 1\le a<\frac tp<\frac s4,
\]
which is a valid representation \eqref{eq:twoprime-event}.

Finally, the density of integers divisible by any bad admissible pair is, by a finite union
bound,
\[
  \le\sum_{p\in\mathcal H}\frac1p\sum_{q\text{ bad for }p}\frac1q
  \;\ll\; \ell^{-3}\sum_{p\in\mathcal H}\frac1p \;\ll\; \ell^{-3}.
\]
Mertens' theorem bounds the last prime reciprocal sum uniformly. No distribution theorem
for primes in arithmetic progressions is used.
\end{proof}

\subsection{An exact finite divisibility probability}

Put
\[
  A=\prod_{p\in\mathcal H}\Bigl(1-\frac1p\Bigr), \qquad
  V_p=\prod_{B(p)<q\le L_0}\Bigl(1-\frac1q\Bigr),
\]
with prime indices. Let $\mathcal P(t,s)$ be the density of integers divisible by at least
one admissible pair.

\begin{lemma}\label{lem:prob}
\begin{equation}\label{eq:prob}
  \mathcal P(t,s)=1-A\Bigl(1+\sum_{p\in\mathcal H}\frac{V_p}{p-1}\Bigr).
\end{equation}
Consequently $U_2^{\mathrm{pr}}(t,s)\ge\mathcal P(t,s)-O(\ell^{-3})$ uniformly on $J$.
\end{lemma}

\begin{proof}
An integer with at least two distinct prime divisors in $\mathcal H$ has an admissible pair,
since $B(p)\le L_0$. An integer with exactly one such prime divisor $p$ has an admissible
pair precisely when it has some prime divisor in $(B(p),L_0]$. The complement therefore
consists of: no prime divisor in $\mathcal H$, of density $A$; and exactly one such prime
$p$ with no prime divisor in $(B(p),L_0]$, of density $A V_p/(p-1)$. These cases are
disjoint, and finite prime divisibility conditions are independent by the Chinese remainder
theorem. Proposition~\ref{prop:bad} gives the last claim.
\end{proof}

\subsection{The limiting profile}

For $0\le w\le 2$, write $h(w)=(w+2)/2$ and $M(w)=\max(1,w)$, and define
\begin{equation}\label{eq:G}
  G(w)=1-\frac{M(w)+\displaystyle\int_{M(w)}^{h(w)}\frac{\max(1,2x-2)}{x}\,\mathrm{d}x}{h(w)}.
\end{equation}
Here $1\le M(w)\le h(w)\le 2$. On the interval of integration, $\max(1,2x-2)\le x$, so
$G\ge 0$, with strict positivity on $(0,2)$.

\begin{proposition}\label{prop:profile}
Uniformly for $\kappa\in J$,
\[
  U_2^{\mathrm{pr}}(t,s) \;\ge\; G\Bigl(\frac{1-\kappa}{\kappa}\Bigr)-o(1).
\]
\end{proposition}

\begin{proof}
Write $a=\max(\kappa,1-\kappa)$ and $b=(1+\kappa)/2$. The endpoints in \eqref{eq:windows}
have logarithmic exponents $a+o(1)$ and $b+o(1)$. Mertens' prime-product estimate gives
$A=a/b+o(1)$. For $p\in\mathcal H$, put $u=\log p/\log t$. Since all lower endpoints are
positive powers of $t$, uniformly on $J$,
\[
  V_p=\frac{\max(\kappa,2u-2\kappa)}{a}+o(1).
\]
Constants in the endpoints contribute $O(1/\ell)$ to these exponents, and the logarithmic
factors contribute $O(\log\ell/\ell)$. Replacing $1/(p-1)$ by $1/p$ costs $o(1)$. Mertens'
reciprocal-sum estimate and partial summation now imply
\[
  \sum_{p\in\mathcal H}\frac{V_p}{p-1}
  =\frac1a\int_a^b\frac{\max(\kappa,2u-2\kappa)}{u}\,\mathrm{d}u+o(1),
\]
uniformly because the piecewise-smooth integrands have uniformly bounded variation on these
compact ranges. Insert these estimates into \eqref{eq:prob} and set
$w=(1-\kappa)/\kappa$, $x=u/\kappa$. Then $a/\kappa=M(w)$ and $b/\kappa=h(w)$, giving
exactly $G(w)$. Apply Lemma~\ref{lem:prob}.
\end{proof}

\subsection{Summing over practical components}

For fixed $0<\alpha<\beta<1$, the finite-complement formula \eqref{eq:tail} rearranges to
$\sum_{y\le\sigma(m)<z}\alpha_m=\nu_y-\nu_z$ for integers $y<z$. Combined with
\eqref{eq:lower} this yields
\begin{equation}\label{eq:mass}
  \sum_{t^\alpha\le\sigma(m)<t^\beta}\alpha_m
  =\frac{K}{\log t}\Bigl(\frac1\alpha-\frac1\beta\Bigr)
  +O\!\Bigl(\frac{\log\log t}{(\log t)^2}\Bigr).
\end{equation}
Consequently, for every continuous $F$ on a fixed compact interval $J\subset(1/3,1)$,
\begin{equation}\label{eq:measure}
  (\log t)\sum_{\log\sigma(m)/\log t\in J}\alpha_m
  F\Bigl(\frac{\log\sigma(m)}{\log t}\Bigr)
  \longrightarrow K\int_J F(\kappa)\,\frac{\mathrm{d}\kappa}{\kappa^2}.
\end{equation}
This follows from a finite partition of $J$, applying \eqref{eq:mass} to each half-open
cell, and squeezing with upper and lower step functions. Endpoint masses tend to zero after
multiplication by $\log t$, also by \eqref{eq:mass} on shrinking fixed neighbourhoods. No
countable additivity of natural density is used.

\begin{proof}[Proof of Theorem~\ref{thm:pointwise}]
Apply \eqref{eq:reduction}, restrict to components with exponent in $J$, and use
Proposition~\ref{prop:profile}. The uniform $o(1)$ costs $o(1)$ after multiplication by
$\log t$, because the total component mass on $J$ is $O(1/\log t)$. Formula
\eqref{eq:measure} gives
\[
  \liminf_{t\to\infty}(\log t)\,(d_t-\nu_t)
  \;\ge\; K\int_J G\Bigl(\frac{1-\kappa}{\kappa}\Bigr)\frac{\mathrm{d}\kappa}{\kappa^2}.
\]
Exhaust $(1/3,1)$ by compact intervals. Nonnegativity of $G$ permits monotone convergence.
The substitution $w=1/\kappa-1$, with $\mathrm{d}w=-\mathrm{d}\kappa/\kappa^2$, gives the
lower bound $KI$ with
\begin{equation}\label{eq:I}
  I=\int_0^2 G(w)\,\mathrm{d}w.
\end{equation}
Finally $(\log t)\nu_t\to K$ by \eqref{eq:lower}. The decimal consequence is justified in
the next section.
\end{proof}

\begin{proof}[Proof of Corollary~\ref{cor:average}]
Theorem~\ref{thm:pointwise} supplies $\eta>0$ such that $(\log t)\,(d_t-\nu_t)\ge\eta$ for
all large $t$. Averaging over $[X,2X)$ gives the dyadic bound with $\beta_0=\eta/2$, and
the exceptional set in that interval is empty for large $X$.
\end{proof}

\begin{proof}[Proof of Corollary~\ref{cor:refuted-constant}]
By Corollary~\ref{cor:exponent}, $\delta<c_4\le 1$. If $c_4=1$ then $(\log t)\,d_t$
converges, say to $c_3$, and Theorem~\ref{thm:pointwise} forces $c_3\ge K(1+I)>0.79$.
\end{proof}

\section{The constant}\label{sec:constant}

The function \eqref{eq:G} is explicitly
\begin{equation}\label{eq:pieces}
G(w)=
\begin{cases}
1-\dfrac{2}{w+2}\left(1+\log\dfrac{w+2}{2}\right),&0\le w\le 1,\\[10pt]
\dfrac{4-3w-2\log(3/(2w))+4\log((w+2)/3)}{w+2},&1\le w\le 3/2,\\[10pt]
\dfrac{w-2+4\log((w+2)/(2w))}{w+2},&3/2\le w\le 2.
\end{cases}
\end{equation}
The three formulas agree at both junctions. The first piece integrates in closed form to
$1-2\log(3/2)-\log^2(3/2)$. High-precision quadrature gives
$I=0.058879779905129\ldots$.

Weingartner \cite[eq.~(17)]{We20} records $1.33607322<c<1.33607654$, whence
$K(1+I)\approx 0.79432$. For the stated decimal theorem it suffices to certify $I>4/75$,
since $K>3/4$ follows from $c>1.33607322>(3/4)e^{\gamma}$. Split each of the three
intervals in \eqref{eq:pieces} into $1000$ equal rational subintervals. For every logarithm
of a rational $x\ge 1$, the series
\[
  2\sum_{j=0}^{11}\frac{z^{2j+1}}{2j+1}
  \;\le\; \log x \;\le\;
  2\sum_{j=0}^{11}\frac{z^{2j+1}}{2j+1}+\frac{2z^{25}}{25(1-z^2)},
  \qquad z=\frac{x-1}{x+1},
\]
gives rational bounds. On each cell $[a,b]$ a lower bound for $G$ follows by monotonicity
of each displayed elementary term (and $G\ge 0$). Rounding each enclosed bound down to
$12$ decimal places in exact rational arithmetic and summing against the cell width yields
\[
  I \;\ge\; \frac{116470605567529}{2000000000000000}
  = 0.0582353027837645 > \frac4{75}.
\]
Hence $K(1+I)>(3/4)(1+4/75)=79/100$, which yields $d_t\ge 0.79/\log t$ for all
sufficiently large $t$ by the definition of liminf.

Erd\H{o}s's question \eqref{eq:erdos-conj} remains open. The constraints are now
$\delta<c_4\le 1$, and if $c_4=1$ then the constant exceeds $0.79$. No $O(1/\log t)$ upper
bound, no limiting constant, and no numerical threshold for ``sufficiently large'' are
asserted.

\end{document}